\documentclass[11pt,xcolor=dvipsnames,svgnames,table,reqno]{amsart}

\usepackage[left=1.5in, right=1.5in, bottom=1.25in,
height=8in]{geometry}

\input xy
\xyoption{all}
\usepackage{accents}
\usepackage{epsfig}
\usepackage{xcolor}
\usepackage{amsthm}
\usepackage{amssymb}
\usepackage{bbm}
\usepackage{amsmath}
\usepackage{amscd}
\usepackage{amsopn}
\usepackage{graphicx}
\usepackage{xspace}

\usepackage{hhline}
\usepackage{easybmat}
\usepackage{caption}   
\usepackage{relsize}

\usepackage{url}
\usepackage{enumitem, hyperref}\hypersetup{colorlinks}

\usepackage{stmaryrd}

\usepackage[T1]{fontenc}

\colorlet{purpleB70}{blue!70!red}

\colorlet{orangeR65}{red!65!yellow}

\definecolor{red2}{HTML}{d41173}

\definecolor{neongreen}{HTML}{1bf702}

\definecolor{radicalred}{HTML}{FF355E}

\definecolor{denim}{HTML}{1560BD}

\definecolor{darkcyan}{rgb}{0.0, 0.55, 0.55}

\definecolor{cilek}{HTML}{FF43A4}

\definecolor{mor}{HTML}{9F00C5}

\definecolor{phlox}{rgb}{0.87, 0.0, 1.0}

\definecolor{fluorescentpink}{HTML}{FF1493}

\definecolor{napiergreen}{rgb}{0.16, 0.5, 0.0}

\definecolor{kellygreen}{rgb}{0.3, 0.73, 0.09}

\definecolor{parisgreen}{HTML}{ 50C878 }

\definecolor{palatinateblue}{rgb}{0.15, 0.23, 0.89}

\definecolor{ceruleanblue}{rgb}{0.16, 0.32, 0.75}

\definecolor{brandeisblue}{rgb}{0.0, 0.44, 1.0}

\definecolor{KLMblue}{HTML}{0FC0FC}

\definecolor{cinnamon}{rgb}{0.82, 0.41, 0.12}

\definecolor{darkorange}{rgb}{1.0, 0.55, 0.0}

\definecolor{darktangerine}{rgb}{1.0, 0.66, 0.07}

\definecolor{deepcarrotorange}{rgb}{0.91, 0.41, 0.17}

\definecolor{internationalorange}{HTML}{FF4F00}

\definecolor{persimmon}{HTML}{EC5800}

\definecolor{pumpkin}{HTML}{FF7518}

\definecolor{darkred}{rgb}{1,0,0} 
\definecolor{darkgreen}{rgb}{0,0.7,0}
\definecolor{darkblue}{rgb}{0,0,1}

\hypersetup{colorlinks,
linkcolor=darkblue,
filecolor=darkgreen,
urlcolor=darkred,
citecolor=darkgreen}

\makeatletter
\def\reflb#1#2{\begingroup
    #2%
    \def\@currentlabel{#2}%
    \phantomsection\label{#1}\endgroup
}
\makeatother

\numberwithin{equation}{section}
\newtheorem{Theorem}{Theorem}
\numberwithin{Theorem}{section}

\newtheorem   {Lemma}[Theorem]{Lemma}

\newtheorem   {Proposition}[Theorem]{Proposition}
\newtheorem   {Corollary}[Theorem]{Corollary}
\theoremstyle {definition}

\theoremstyle {remark}
\newtheorem   {Remark}[Theorem]{Remark}
\newtheorem   {Example}[Theorem]{Example}

\theoremstyle{conjecture} 
\newtheorem*  {Conjecture*}{Conjecture}  

\def    \eps    {\epsilon}

\newcommand{\CU}{{\mathcal U}}

\newcommand{\CS}{{\mathcal S}}

\newcommand{\supp}{\operatorname{supp}}

\newcommand{\Ham}{{\mathit{Ham}}}

\newcommand{\id}{{\mathit id}}

\newcommand{\const}{{\mathit const}}

\newcommand{\tH}{\tilde{H}}

\newcommand{\CV}{{\mathcal V}}

\newcommand{\CW}{{\mathcal W}}

\def    \F      {{\mathbb F}}

\def    \R      {{\mathbb R}}

\def    \Z      {{\mathbb Z}}
\def    \N      {{\mathbb N}}
\def    \Q      {{\mathbb Q}}

\def    \T      {{\mathbb T}}
\def    \CP     {{\mathbb C}{\mathbb P}}

\def    \12     {{\frac{1}{2}}}

\def    \im     {\operatorname{im}}

\def    \HF     {\operatorname{HF}}

\def    \H      {\operatorname{H}}

\def    \CF      {\operatorname{CF}}

\def    \Fix     {\operatorname{Fix}}
\def    \Per     {\operatorname{Per}}

\def    \s     {\operatorname{c}}

\def    \hn    {\scriptscriptstyle{H}}

\def    \inv   {\mathrm{inv}}

\newcommand    \htop  {\operatorname{h_{\scriptscriptstyle{top}}}}

\newcommand   \ugamma   {\underline{\gamma}}

\begin{document}


\setlength{\smallskipamount}{6pt}
\setlength{\medskipamount}{10pt}
\setlength{\bigskipamount}{16pt}





\title [Generic behavior of the spectral norm]{On the generic behavior
  of the spectral norm}

\author[Erman \c C\. inel\. i]{Erman \c C\. inel\. i}
\author[Viktor Ginzburg]{Viktor L. Ginzburg}
\author[Ba\c sak G\"urel]{Ba\c sak Z. G\"urel}

\address{E\c C: Institut de Math\'ematiques de Jussieu - Paris Rive
  Gauche (IMJ-PRG), 4 place Jussieu, Boite Courrier 247, 75252 Paris
  Cedex 5, France} \email{erman.cineli@imj-prg.fr}

\address{VG: Department of Mathematics, UC Santa Cruz, Santa
  Cruz, CA 95064, USA} \email{ginzburg@ucsc.edu}

\address{BG: Department of Mathematics, University of Central Florida,
  Orlando, FL 32816, USA} \email{basak.gurel@ucf.edu}

\subjclass[2020]{53D40, 37J11, 37J46} 

\keywords{Spectral norm, Periodic points, Hamiltonian diffeomorphisms,
  Floer homology}

\date{\today} 

\thanks{The work is partially supported by NSF grants DMS-2304207 (BG)
  and DMS-2304206 (VG), Simons Foundation Travel Support 855299 (BG)
  and MP-TSM-00002529 (VG), and ERC Starting Grant 851701 via a
  postdoctoral fellowship (E\c{C})}



\begin{abstract}
  The main result of the paper is that for any closed symplectic
  manifold the spectral norm of the iterates of a Hamiltonian
  diffeomorphism is locally uniformly bounded away from zero
  $C^\infty$-generically.
\end{abstract}

\maketitle

\tableofcontents

\section{Introduction}
\label{sec:intro}
We show that for a Hamiltonian diffeomorphism $\varphi$ of a closed
symplectic manifold $M$ the spectral norm over $\Q$ of the iterates
$\varphi^k$ is locally uniformly bounded away from zero
$C^\infty$-generically in $\varphi$, without any additional
assumptions on $M$.

The question of the behavior of the sequence
$\gamma\big(\varphi^k\big)$ of spectral norms goes back to the work of
Polterovich, \cite{Po}. Recently, there has been renewed interest in
the problem whether and when this sequence is bounded away from
zero. There are several reasons for this question, amounting roughly
speaking to the fact that one can obtain pretty strong results on the
symplectic dynamics of $\varphi$ when the sequence is \emph{not}
bounded away from zero:
\begin{equation}
\label{eq:gamma-ai}
\ugamma(\varphi):=\liminf_{k\to\infty} \gamma\big(\varphi^k\big)=0.
\end{equation}
Among these are, for instance, Lagrangian Poincar\'e recurrence,
\cite{GG:PR, JS}, and the variant of the strong closing lemma from
\cite{CS}. Simultaneously, fairly explicit criteria for this sequence
to be bounded away from zero have been established, based on the
crossing energy theorem from \cite{GG:hyperbolic, GG:PR}; see, e.g.,
\cite{CGG:Growth} and Theorem \ref{thm:gamma}. Let us now provide some
more context for the question.

First, note that the condition \eqref{eq:gamma-ai} can be interpreted
as that $\varphi$ is $\gamma$-rigid or, in other words, a
$\gamma$-approximate identity.

This notion is a particular case of a much more general
concept. Namely, consider a class of diffeomorphisms $\varphi$ or even
homeomorphisms of a manifold $M$, which we assume here to be closed.
For instance, this can be the class of all diffeomorphisms or of
Hamiltonian diffeomorphisms when $M$ is symplectic, etc. Assume
furthermore that this class is equipped with some norm $\|\cdot\|$,
e.g., the $C^0$- or $C^1$-norm or the $\gamma$- or Hofer-norm in the
Hamiltonian case.  A map $\varphi$ is said to be
$\|\cdot\|$-\emph{rigid} if $\varphi^{k_i}\to \id$ with respect to
$\|\cdot\|$, i.e., $\big\|\varphi^{k_i}\big\|\to 0$, for some sequence
$k_i\to\infty$. The term ``rigid'' is somewhat overused in dynamics
and also frequently confused with structural stability, and in
\cite{GG:AI} we proposed to call such a map $\varphi$ a
$\|\cdot\|$-\emph{approximate identity}, or a $\|\cdot\|$-\emph{a.i.}\
for the sake of brevity. We refer the reader to, e.g., \cite{Br,
  GG:AI, CS} for a further discussion of approximate identities, aka
rigid maps, in different contexts and further references. Here we only
mention that $C^r$-a.i.\ is obviously $C^s$-a.i.\ for any $s\leq r$
and, when $M$ is aspherical or $M=\CP^n$, a $C^0$-a.i.\ is also a
$\gamma$-a.i.; see \cite{BHS, Sh:V}.

Zeroing in on $\gamma$-a.i.'s we note that there are rather few
examples of such maps. The most dynamically interesting examples are
Hamiltonian pseudo-rotations. This class of maps has been extensively
studied in a variety of settings by dynamical systems methods and more
recently from the perspective of symplectic topology and Floer theory;
see, e.g., \cite{AK, A-Z, Br, FK, GG:PR, JS, LRS} and references
therein.

While the official definitions of Hamiltonian pseudo-rotations vary,
these are, roughly speaking, Hamiltonian diffeomorphisms with a finite
and minimal possible number of periodic points (in the sense of
Arnold's conjecture); see \cite{GG:PR, Sh:JMD, Sh:MRL}. For instance,
when $M=\CP^n$ this number is $n+1$. Most likely, for many symplectic
manifolds this condition can be relaxed. Namely, in all examples of
Hamiltonian diffeomorphisms $\varphi$ with finitely many periodic
points, all periodic points are fixed points and their number is
minimal possible.  Thus $\varphi$ is a pseudo-rotation. For a certain
class of manifolds $M$, including $\CP^n$, this has been established
rigorously under a minor non-degeneracy assumption; see \cite{Sh:HZ}
and also \cite{CGG:HZ}. Moreover, in all examples to date of
Hamiltonian diffeomorphisms $\varphi$ with finitely many periodic
points, $\varphi$ is non-degenerate.

In general, the relation between pseudo-rotations and $\gamma$-a.i.'s
is not obvious. All known Hamiltonian pseudo-rotations are
$\gamma$-a.i.'s and for $M=\CP^n$ this is proved in \cite{GG:PR} by
using the results from \cite{GG:gaps}. The converse is not true: for
instance any element $\varphi$ of a Hamiltonian torus action is a
$\gamma$-a.i., although $\varphi$ need not have isolated fixed
points. (It is conceivable that for a strongly non-degenerate
$\gamma$-a.i., the periodic points are necessarily the fixed points:
in the obvious notation, $\Per(\varphi)=\Fix(\varphi)$. However, a map
$\varphi$ with the latter property need not be a $\gamma$-a.i. For
instance, $\gamma\big(\varphi^k\big)$ can grow linearly for such a
map; see Remark \ref{rmk:gamma-unbounded}.)

Most closed symplectic manifolds $(M,\omega)$ admit no
pseudo-rotations, i.e., every Hamiltonian diffeomorphism of $M$ has
infinitely many periodic points. This statement (for a specific
manifold $M$) is usually referred to as the Conley conjecture. To
date, the Conley conjecture has been shown to hold unless there exists
$A\in \pi_2(M)$ such that $\left<[\omega], A\right>>0$ and
$\left<c_1(TM), A\right>>0$; see \cite{Ci, GG:survey, GG:Rev} and
references therein. In particular, the Conley conjecture holds when
$M$ is symplectically aspherical or negative monotone. Furthermore,
for a broad class of closed symplectic manifolds, $\varphi$ has
infinitely many periodic points $C^\infty$-generically; see
\cite{GG:generic, Su21} and Section~\ref{sec:Sugimoto}.

Although the classes of Hamiltonian pseudo-rotations and
$\gamma$-a.i.'s are certainly different, there is a clear parallel
between these two classes and their existence conditions on $M$.

\begin{Conjecture*} Let $M$ be closed symplectic manifold.
  \begin{itemize}
  \item[\reflb{C1}{\rm{(i)}}] The manifold $M$ admits no
    $\gamma$-a.i.'s unless there exists $A\in \pi_2(M)$ such that
    $\left<[\omega], A\right>>0$ and $\left<c_1(TM), A\right>>0$.
  
  \item[\reflb{C2}{\rm{(ii)}}] A Hamiltonian diffeomorphism
    $\varphi\colon \CP^n\to \CP^n$ is a $\gamma$-a.i.\ if and only if
    all iterates $\varphi^k$ are Morse--Bott non-degenerate and
    $\dim \H_*\big(\Fix(\varphi^k); \F\big)=n+1$ for all $k\in\N$ and
    any ground field $\F$.
    
\end{itemize}
\end{Conjecture*}  

This conjecture is supported by some evidence. For instance, $M$ does
not admit periodic Hamiltonian diffeomorphisms $\varphi$ (i.e.,
$\varphi^{N}=\id$ for some $N>1$) when $M$ satisfies the conditions of
\ref{C1}; see \cite{AS,Po}.  In addition, $\Fix(\varphi^k)$ is
Morse--Bott non-degenerate whenever $\varphi$ is periodic. This is a
consequence of the equivariant Darboux lemma; see, e.g., \cite[Thm.\
22.2]{GS}. Moreover, aspherical or negative monotone symplectic
manifolds do not admit $C^1$-a.i.'s; see \cite{Po} and
\cite{Su23}. Further results and references along these lines can be
found in \cite{AS}. In \cite{CGG:Growth} both assertions are proved in
dimension two for strongly non-degenerate Hamiltonian diffeomorphisms;
see Corollary \ref{cor:sphere}.  Moreover, in the setting of \ref{C1}
the sequence of the spectral norms $\gamma\big(\varphi^p\big)$ over
$\Z/p\Z$, where $p$ ranges through all primes, is separated away from
zero, \cite{Sh:private}.  As we have already mentioned the ``if'' part
of \ref{C2} is established in \cite{GG:PR} without any non-degeneracy
assumption when $|\Per(\varphi)|=n+1$. With this in mind, Part
\ref{C2} of the conjecture asserts, in particular, that every
pseudo-rotation of $\CP^n$ is strongly non-degenerate.

\begin{Remark} While Part \ref{C2} of the conjecture might extend to
  some other ambient symplectic manifolds $M$, some restriction on $M$
  is necessary. For instance, the torus $\T^{2n}$ equipped with an
  irrational symplectic structure admits a Hamiltonian diffeomorphism
  $\varphi$ such that the conditions of \ref{C2} are satisfied but
  $\gamma\big(\varphi^k\big)\to \infty$; see \cite{Ze} and also
  \cite{Ci:new} for further constructions of this type with
  complicated dynamics.
\end{Remark}

In a similar vein, the main result of this paper can be thought of as
the $\gamma$-a.i.\ analogue of the aforementioned theorem on the
$C^\infty$-generic Conley conjecture, although at this moment the
proof of the latter requires some additional conditional conditions on
the underlying manifold; see Section \ref{sec:Sugimoto}.

\begin{Remark}
  \label{rmk:behavior}
  Overall, rather little is known about the behavior of the
  $\gamma$-norm under iterations. For a certain class of manifolds,
  including $\CP^n$, the spectral norm is \emph{a priori} bounded from
  above, \cite{EP, KS}. However, such manifolds appear to be rare; see
  Remark \ref{rmk:gamma-unbounded}. Also, the sequence
  $\gamma\big(\varphi^k\big)$ is bounded from above when
  $\supp \varphi$ is displaceable in $M$, but not much beyond these
  facts and the results of this paper is known about the behavior of
  this sequence. For instance, when $M$ is a surface of positive
  genus, it is not known if $\gamma\big(\varphi^k\big)$ necessarily
  grows linearly or can be bounded from above when $\varphi$ is
  strongly non-degenerate or, as the opposite extreme, autonomous and
  $\supp\varphi$ is not displaceable.
\end{Remark}

\begin{Remark}
  It is worth keeping in mind that in contrast with some other
  dynamics concepts, in most if not all settings a.i.'s are sensitive
  to reparametrization. To be more specific, let an a.i.\ $\varphi$ be
  the time-one map of the flow of a vector field $X$ and let $\psi$ be
  the time-one map of $fX$ for some function $f>0$. Then, in general,
  $\psi$ need not be an a.i. For instance, assume that $X$ is a solid
  rotation vector field on $M=S^2$ and $f\neq \const$. Then one can
  show that $\psi$ is not a $C^0$-a.i., and hence not a $C^{r}$-a.i.\
  for any $r\geq 0$. Apparently, the same is true for the
  $\gamma$-norm, but this fact is yet to be proved rigorously; cf.\
  item \ref{C2} of the Conjecture.
\end{Remark}

\medskip\noindent{\bf Acknowledgements.}
Parts of this work were carried out while the second and third authors
were visiting the IMJ-PRG, Paris, France, in May 2023 and also during
the Summer 2023 events \emph{Symplectic Dynamics Workshop} at INdAM,
Rome, Italy, and \emph{From Smooth to $C^0$ Symplectic Geometry:
  Topological Aspects and Dynamical Implications Conference} at CIRM,
Luminy, France. The authors would like to thank these institutes for
their warm hospitality and support.

\section{Preliminaries and notation}
 \label{sec:prelim}
 In this section we very briefly set our notation and conventions
 which are quite standard and spelled out in more detail in, e.g.,
 \cite{CS}. The reader may find it convenient to jump to Section
 \ref{sec:results} and consult this section only as needed.

 Throughout the paper, all manifolds, functions and maps are assumed
 to be $C^\infty$-smooth unless specifically stated otherwise.
 
 Let $(M^{2n},\omega)$ be a closed symplectic manifold. A
 \emph{Hamiltonian diffeomorphism} $\varphi=\varphi_H=\varphi_H^1$ is
 the time-one map of the time-dependent flow $\varphi^t=\varphi_H^t$
 of a 1-periodic in time Hamiltonian $H\colon S^1\times M\to\R$, where
 $S^1=\R/\Z$. We set $H_t=H(t,\cdot)$. The Hamiltonian vector field
 $X_H$ of $H$ is defined by $i_{X_H}\omega=-dH$.  We say that
 $\varphi$ is \emph{non-degenerate} if all fixed points of $\varphi$
 are non-degenerate, and \emph{strongly non-degenerate} if all
 periodic points of $\varphi$ are non-degenerate.  We will denote
   by $ \Ham(M,\omega)$ the group of Hamiltonian diffeomorphisms of
   $(M, \omega)$.

 Recall that the \emph{spectral norm}, also known as the
 $\gamma$-\emph{norm}, of $\varphi$ is defined as
\[
  \gamma(\varphi) = \inf_H\big\{\s(H)+\s\big(H^{\inv}\big)\mid \varphi
  =\varphi_H\big\},
\]
where $H^{\inv}(x)=-H_t\big(\varphi_H^t(x)\big)$ is the Hamiltonian
generating the flow $\big(\varphi_H^{t}\big)^{-1}$ and $\s=\s_{[M]}$
is the spectral invariant associated with the fundamental class
$[M]\in \H_{2n}(M)$. (Here we can take as $H^{\inv}$ any Hamiltonian
generating this flow with the same time/space average as $H$.) The
infimum is taken over all 1-periodic in time Hamiltonians $H$
generating $\varphi$, i.e., $\varphi=\varphi_H$.  The \emph{Hofer
  norm} of $\varphi$ is defined as
$$
\|\varphi\|_{\hn} = \inf_{H}\int_{S^1}\big(\max_M H_t-\min_M
H_t\big)\, dt,
$$
where the infimum is again taken over all 1-periodic in time
Hamiltonians $H$ generating $\varphi$. Then
\[
  \gamma(\varphi)\leq \|\varphi\|_{\hn}.
\]
We refer the reader to, e.g., \cite{Oh:gamma, Oh:constr, Sc, Vi} and
also, e.g., \cite{CS, EP, GG:gaps, KS, Po:Book, Us0, Us1}, for the
original treatment and a detailed discussion of spectral invariants
and these norms, and for further references.

Here we are interested in the behavior of $\gamma\big(\varphi^k\big)$,
$k\in \N$, and in particular in the question when this sequence is
bounded away from zero. As in the introduction, set
\[
  \ugamma(\varphi)=\liminf_{k\to\infty} \gamma\big(\varphi^k\big)\in
  [0,\,\infty].
\]

These definitions implicitly depend on the construction of the
filtered Floer homology $\HF^a(H)$ for the action window
$(-\infty,\,a)$.  In this paper we do not in general assume that the
class $[\omega]$ is rational or that $\varphi$ is
non-degenerate. Hence, we feel, a word is due on the specifics of the
definitions.

Assume first that $H$ is non-degenerate. Then we utilize Pardon's VFC
package, \cite{Pa}, to define the filtered Floer homology $\HF^a(H)$
over $\Q$ and spectral invariants; see, e.g., \cite{CS, Us0}. To be
more specific, $\HF^a(H)$ is the homology of the subcomplex $\CF^a(H)$
of the Floer complex $\CF(H)$ generated by Floer chains with action
below $a$. Virtually any choice of the \emph{Novikov field} can be
used here. We take the standard Novikov field
\[
  \Lambda=\big\{\sum_{A\in \Gamma} b_A A\,\big| \,b_A\in \Q \textrm{
    and } \#\{b_A\neq 0,\, \omega(A)>c\}<\infty\,\, \forall\, c\in \R
  \big\},
\]
where $\Gamma=\pi_2(M)/\big(\ker[\omega]\cap \ker
c_1(TM)\big)$. Alternatively, we could have used the universal Novikov
field. Then, for any $\alpha\in \H_*(M)\otimes \Lambda$, the spectral
invariant $\s_\alpha(H)$ is defined as
\begin{equation}
  \label{eq:spec-inv}
\s_\alpha(H)=\inf\{a\in \R \mid \alpha \in \im \iota_a\},
\end{equation}
where
\begin{equation}
  \label{eq:iota}
\iota_a\colon \HF^a(H)\to \HF(H)\cong \H_*(M)\otimes \Lambda
\end{equation}
is the natural inclusion-induced map and the identification on the
right is the PSS-isomorphism. We note that all spectral invariants
necessarily belong to the action spectrum $\CS(H)$ of $H$ when $H$ is
non-degenerate, \cite{Us0}.

When $H$ is not necessarily non-degenerate, we set
\[
  \s_\alpha(H):=\inf_{\tH\geq H} \s_\alpha(\tH)= \sup_{\tH\leq H}
  \s_\alpha(\tH) =\lim_{\tH\to H}\s_\alpha(\tH),
\]
where $\tH$ is non-degenerate and the convergence $\tH\to H$ is taken
to be $C^0$. The second and third equalities and the existence of the
limit follow from that $\s_\alpha$ is monotone and
$\s_\alpha(\tH+\const)=\s_\alpha(\tH)+\const$. Alternatively, we could
have set
\[
  \HF^a(H)=\varinjlim_{\tH\geq H}\HF^a(\tH),
\]
and then used \eqref{eq:spec-inv} and \eqref{eq:iota} to get the same
result.

Defined in this way, spectral invariants $\s_\alpha$ can be easily
shown to have all the standard properties: $\s_\alpha(H)$ is monotone
and Lipschitz continuous in $H$ with Lipschitz constant one;
$\s_\alpha(H+\const)=\s_\alpha(H)+\const$; etc. (We refer the reader
to, e.g., \cite{CS} for more details.) The exception is that
$\s_\alpha(H)$ has been proven to be spectral, i.e., an element of
$\CS(H)$, only when $[\omega]$ is rational or $H$ is non-degenerate;
see \cite{EP, Oh:constr, Us0}.

\section{Main results}
\label{sec:results} 

The key to bounding $\ugamma$ from below is the following fact
connecting the behavior of $\gamma\big(\varphi^k\big)$ with the
dynamics of $\varphi$ and, in particular, its hyperbolic points.

\begin{Theorem}
  \label{thm:gamma}
  Let $\varphi\colon M\to M$ be a Hamiltonian diffeomorphism of a
  closed symplectic manifold $M$ with more than $\dim \H_*(M)$
  hyperbolic periodic points.  Then $\ugamma(\varphi)>0$. Moreover,
  $\ugamma$ is locally uniformly bounded away from zero near
  $\varphi$, i.e., 
  there exists $\delta>0$, possibly depending on $\varphi$, and
  a sufficiently $C^\infty$-small neighborhood $\CU$ of $\varphi$ such
  that
  \[
    \ugamma(\psi)>\delta \textrm{ for all } \psi\in \CU.
  \]
\end{Theorem}

Without the moreover part, this theorem was originally proved in
\cite{CGG:Growth}. We give a complete proof in Section
\ref{sec:proofs}. Let us emphasize that in Theorem \ref{thm:gamma} we
impose no non-degeneracy requirements on $\varphi$, and also that the
property of $\varphi$ to have more than $\dim \H_*(M)$ hyperbolic
periodic points, or more than any fixed number of hyperbolic periodic
points, is open in $C^1$-topology.

\begin{Example}
  Assume that $M$ is a closed surface and $\htop(\varphi)>0$. Then
  $\varphi$ has infinitely many hyperbolic periodic points,
  \cite{Ka}. Hence, $\ugamma(\varphi)>0$. Moreover,
  $\ugamma(\psi)>\delta$ for some $\delta>0$ and all $\psi$ which are
  $C^\infty$-close to $\varphi$. Also note in connection with Theorem
  \ref{thm:generic} and Corollary \ref{cor:sphere} below that
  $\htop>0$ is a $C^\infty$-generic condition in dimension two,
  \cite{LCS}.
\end{Example}

The requirement of the theorem that the number of hyperbolic points is
greater $\dim \H_*(M)$ can be further relaxed by looking only at the
odd/even-degree homology of $M$, depending on whether $n=\dim M/2$ is
odd or even; see Remark \ref{rmk:odd}.

The main result of the paper is the following theorem relying on
Theorem~\ref{thm:gamma} and the strong closing lemma from \cite{CS}.

\begin{Theorem}
\label{thm:generic}
Let $M$ be a closed symplectic manifold.  The function $\ugamma$ is
locally uniformly bounded away from zero on a $C^\infty$-open and
dense set of Hamiltonian diffeomorphisms $\varphi\colon M\to M$, i.e.,
for every $\varphi$ in this set there exists $\delta>0$, possibly
depending on $\varphi$ but not on $\psi$, such that
\[
  \ugamma(\psi)>\delta
\]
whenever $\psi$ is sufficiently $C^\infty$-close to $\varphi$.
\end{Theorem}

We note that we do not assert here that in general the set of
Hamiltonian diffeomorphisms $\varphi$ with $\ugamma(\varphi)>0$ is
itself $C^\infty$-open, but rather that this set contains a set which
is $C^\infty$-open and dense. Nor do we impose any restrictions on the
(symplectic) topology of $M$ or require any of the iterates
$\varphi^k$ to be non-degenerate. The proof of Theorem
\ref{thm:generic} given in Section \ref{sec:pf-prop-gamma} is based on
a variant of the Birkhoff--Lewis--Moser theorem. The key new
ingredient of the proof is the strong closing lemma from \cite{CS}.
It is also worth pointing out that if we replaced the statement that
the set is $C^\infty$-dense by that it is $C^1$-dense, the theorem
would turn into an easy consequence of already known facts; see Remark
\ref{rmk-cinfty}.
  
In several situations, Theorem \ref{thm:generic} can be made slightly
more precise. For instance, we have the following result, also
originally proved in \cite{CGG:Growth} without the moreover part.

\begin{Corollary}
  \label{cor:sphere}
  Assume that $M$ is a surface and $\varphi$ is strongly
  non-degenerate. Then $\ugamma(\varphi)>0$ when $M$ has positive
  genus. When $M$ is the two-sphere, $\ugamma(\varphi)=0$ if and only
  if $\varphi$ is a pseudo-rotation. Moreover, $\ugamma$ is locally
  uniformly bounded from 0 on the set of all strongly non-degenerate
  Hamiltonian diffeomorphisms $\varphi$ when $M$ has positive genus
  and on the set of such $\varphi$ with at least three fixed points
  when $M=S^2$.
\end{Corollary}

\begin{proof}
  When $M$ has positive genus, a Conley conjecture type argument
  guarantees that $\varphi$ has infinitely many hyperbolic periodic
  points; see \cite{FH, GG:survey, SZ} or \cite{LCS}. Thus, in this
  case, the statement follows directly from Theorem \ref{thm:gamma}.

  Concentrating on $M=S^2$, first note that for all, not necessarily
  non-degenerate, pseudo-rotations of $\CP^n$, the sequence
  $\gamma\big(\varphi^k\big)$ contains a subsequence converging to
  zero, and hence $\ugamma(\varphi)=0$; see \cite{GG:PR}. In the
  opposite direction, when $M=S^2$, the existence of one positive
  hyperbolic periodic point is enough to ensure that
  $\ugamma(\varphi)>0$ and, moreover, $\ugamma$ is locally uniformly
  bounded away from zero; see Remark \ref{rmk:odd}. Hence, more
  generally, without any non-degeneracy assumption, if
  $\ugamma(\varphi)=0$, then all periodic points of $\varphi$ are
  elliptic. For strongly non-degenerate Hamiltonian diffeomorphisms
  $\varphi$, this forces $\varphi$ to be a pseudo-rotation.
\end{proof}

Since the Hofer norm is bounded from below by the spectral norm, we
have the following.

\begin{Corollary}
  In all results from this section, we can replace the spectral norm
  by the Hofer norm.
\end{Corollary}

We refer the reader to the next section for further refinements of
Theorems \ref{thm:gamma} and \ref{thm:generic}.

\begin{Remark}
  Throughout the paper all homology groups are taken over $\Q$.  This
  choice of the background coefficient field is necessitated by the
  use of Floer theory for an arbitrary closed symplectic manifold
  $M$. When $M$ is weakly monotone, $\Q$ can be replaced by any
  coefficient field.
\end{Remark}

\section{Proofs and refinements}
\label{sec:proofs}

In section \ref{sec:pf-prop-gamma}, we prove Theorems \ref{thm:gamma}
and \ref{thm:generic}. In Section \ref{sec:Sugimoto}, we refine the
latter result under certain additional assumptions on $M$ and further
comment on the class of $\gamma$-a.i.'s.

\subsection{Proofs of Theorems \ref{thm:gamma} and \ref{thm:generic}}
\label{sec:pf-prop-gamma}

\begin{proof}[Proof of Theorem \ref{thm:gamma}]
  By the conditions of the theorem, for some $N\in \N$, the
  Hamiltonian diffeomorphism $\varphi$ has more than $\dim \H_*(M)$
  hyperbolic $N$-periodic points. We denote the set of these points by
  $\mathcal{K}$. Thus $|\mathcal{K}|>\dim \H_*(M)$ and clearly
  $\mathcal{K}$ is a locally maximal hyperbolic set. Furthermore,
  every point in $\mathcal{K}$ is also $\ell N$-periodic for all
  $\ell\in \N$. For $\eps>0$, denote by $b_\eps(\varphi)$ the number
  of bars in the barcode of $\varphi$ of length greater than $\eps$
  including infinite bars; see, e.g., \cite{CGG:Entropy}. Then,
    we claim that for a sufficiently small $\eps>0$ and any
    $\ell\in \N$,
  \begin{equation}  
  \label{eq:b_eps}
  b_\eps\big(\varphi^{\ell N}\big) \geq
  \dim \H_*(M) + \lceil (|\mathcal{K}| - \dim \H_*(M))/2 \rceil
  >  \dim \H_*(M).
  \end{equation}
  In particular, $\varphi^{\ell N}$ has at least one finite bar of
  length greater than $\eps>0$.

  This inequality is essentially a consequence of \cite[Prop.\ 3.8 and
  6.2]{CGG:Entropy}. To prove \eqref{eq:b_eps}, first note that the
  number of infinite bars in the barcode of any Hamiltonian
  diffeomorphism is equal to $\dim \H_*(M)$. Secondly, it follows from
  \cite[Prop.\ 6.2]{CGG:Entropy} and the proof of \cite[Prop.\
  3.8]{CGG:Entropy} that every periodic point in $\mathcal{K}$ appears
  as an ``end point'' of a bar of length greater than
  $\eps>0$. Combining these two facts, we conclude that
  $\varphi^{\ell N}$ has at least
  $\lceil (|\mathcal{K}| - \dim \H_*(M))/2 \rceil$ finite bars of
  length greater than $\eps>0$, and \eqref{eq:b_eps} follows.

  Furthermore, since the crossing energy lower bound in \cite[Thm.\
  6.1]{CGG:Entropy} is stable under $C^\infty$-small perturbations of
  the Hamiltonian, for every positive $\eta<\eps$ the same is true for
  any Hamiltonian diffeomorphism $\Psi$ which is $C^\infty$-close to
  $\varphi^N$. Namely,
  \[
    b_\eta\big(\Psi^\ell\big)>\dim \H_*(M),
  \]
  and hence the barcode of $\Psi^\ell$ has a finite bar of length
  greater than $\eta$.

  Also, recall that as is proved in \cite[Thm.\ A]{KS}, for any
  $\varphi$,
  \[
    \beta_{\max}(\varphi)\leq \gamma(\varphi),
  \]
  where the left-hand side is the \emph{boundary depth}, i.e., the
  longest finite bar in the barcode of $\varphi$. Thus, for a
  sufficiently small $\eta>0$,
  \begin{equation}
    \label{eq:eps-beta-gamma}
    \eta< \beta_{\max}\big(\Psi^{\ell}\big) \leq \gamma
    \big(\Psi^{\ell }\big).
  \end{equation}
  
  Next, set $\delta=\eta/2$ and arguing by contradiction, assume that
  there exist $\psi$ sufficiently $C^\infty$-close to $\varphi$ and a
  sequence $k_i\to\infty$ such that
\[
  \gamma\big(\psi^{k_i}\big)<\delta.
\]
Since the sequence $k_i$ is infinite and there are only finitely many
residues modulo $N$, there exists a pair $k_i<k_j$ such that
\[
  k_j-k_i=\ell N
\]
for some $\ell\in\N$.

Clearly, $\Psi=\psi^N$ is $C^\infty$ close to $\varphi^N$ when $\psi$
is sufficiently $C^\infty$-close to $\varphi$, and hence
\eqref{eq:eps-beta-gamma} holds. Then by the triangle inequality for
$\gamma$, we have
\[
  \eta< \gamma\big(\Psi^{\ell}\big) \leq \gamma\big(\psi^{k_j}\big) +
  \gamma\big(\psi^{-k_i}\big)<2\delta=\eta.
\]
This contradiction concludes the proof of the theorem.
\end{proof}

\begin{Remark} It might be worth a second to examine how the
  invariants of $\varphi$ involved in the proof depend on the isotopy
  $\varphi^t_H$ in $\Ham(M, \omega)$ generated by $H$ and its lift to
  the universal covering of the group. Namely, $\gamma(\varphi)$ is
  \emph{a priori} independent of the isotopy only on the universal
  covering. On $\Ham(M,\omega)$ it is defined by passing to the
  infimum over often infinitely many elements. However, the boundary
  depth $\beta_{\max}$ is well-defined on $\Ham(M,\omega)$. In the
  proof we bound $\beta_{\max}(\varphi)$ from below (see, e.g.,
  \cite{Us1}) and that bounds $\gamma(\varphi)$ from below regardless
  of the lift, \cite{KS}.
\end{Remark}

\begin{Remark}
  \label{rmk:odd}
  When $n=\dim M/2$ is odd, it is sufficient to require in Theorem
  \ref{thm:gamma} that the number of hyperbolic periodic points is
  greater than $b=\dim\H_{\scriptscriptstyle{odd}}(M)$. For instance,
  this is the case when $M$ is a surface. Indeed, in the proof of the
  theorem by taking $N$ even and sufficiently large, we can guarantee
  that the number of positive hyperbolic $N$-periodic points is
  greater than $b$. Such points necessarily have even Conley--Zehnder
  index, and hence contribute to the odd-degree homology of $M$ under
  the isomorphism $\HF_{*}(\varphi^N)\cong\H_{*+n}(M)$. Likewise, when
  $n$ is even, it suffices to require the number of hyperbolic
  periodic points to be greater than
  $\dim\H_{\scriptscriptstyle{even}}(M)$.
\end{Remark}

\begin{proof}[Proof of Theorem \ref{thm:generic}]
  To prove the theorem, it suffices to show that every $C^\infty$-open
  set $\CU$ in the group of Hamiltonian diffeomorphisms contains an
  open subset $\CW$ such that $\ugamma(\varphi)>\delta$ for all
  $\varphi\in\CW$ and some $\delta=\delta(\CW)>0$ independent of
  $\varphi$. Indeed, then fixing $\CW$ for every $\CU$ we can take the
  union of sets $\CW$ for all $\CU$ as the desired open and dense
  subset.

  Let $q=\dim \H_*(M)$. For any $\CU$, there are two alternatives:
  \begin{itemize}
  \item[(i)] there exists $\varphi\in \CU$ with more than $q$ periodic
    points;
  \item[(ii)] every $\varphi\in \CU$ has at most $q$ periodic points.
  \end{itemize}

  Let us first focus on Case (i). Pick $\varphi\in\CU$ with more than
  $q$ periodic points and fix $q+1$ of them. Denote these points by
  $x_0,\ldots, x_q$, and note that arbitrarily $C^\infty$-close to
  $\varphi$ there exists a Hamiltonian diffeomorphism
  $\varphi'\in \CU$ such that $x_0,\ldots, x_q$ are non-degenerate
  periodic points of $\varphi'$. This is essentially a linear algebra
  fact and to construct $\varphi'$, it suffices to perturb $\varphi$
  near these points, changing $D\varphi$ slightly. (Note that
  $\varphi'$ may have many other periodic points, non-degenerate or
  not. We can ensure in addition that $\varphi'$ is strongly
  non-degenerate, but we do not need this fact.) We replace $\varphi$
  by $\varphi'$, keeping the notation~$\varphi$.

  If all periodic points $x_0,\ldots,x_q$ are hyperbolic, we can take
  as $\CW$ any $C^\infty$-small neighborhood of $\varphi$ by Theorem
  \ref{thm:gamma}.

  If one of the points $x_0,\ldots,x_q$ is not hyperbolic, we argue by
  perturbing $\varphi$ again. Namely, recall that by the
  Birkhoff--Lewis--Moser theorem (see \cite{Mo}), whenever $\varphi$
  has a non-hyperbolic, non-degenerate periodic point $x$, there
  exists an arbitrarily $C^\infty$-small perturbation
  $\varphi'\in \CU$ of $\varphi$ with infinitely many periodic points
  near $x$. Moreover, $\varphi'$ can be chosen so that infinitely many
  of these periodic points are hyperbolic; see \cite[Prop.\
  8.2]{Ar}. (This follows from the proof of the Birkhoff--Lewis--Moser
  theorem.) Thus, again by Theorem \ref{thm:gamma}, we can take a
  sufficiently $C^\infty$-small neighborhood of~$\varphi'$ as $\CW$.

  To deal with Case (ii), we need the following quantitative variant
  of the strong closing lemma.

  \begin{Lemma}[Strong Closing Lemma, \cite{CS}]
    \label{lemma:scl}
    Let $\psi$ be a Hamiltonian diffeomorphism of a closed symplectic
    manifold $M$. Assume that there is a closed ball $V\subset M$
    containing no periodic points of $\psi$, i.e.,
    $V\cap\Per(\psi)=\emptyset$. Let $G\geq 0$ be a Hamiltonian
    supported in $V$ and such that
  \[
    \s(G)>\ugamma(\psi).
  \]
  Then the composition $\psi \varphi_G$ has a periodic orbit passing
  through $V$.
  \end{Lemma}
  
  Pick a non-degenerate Hamiltonian diffeomorphism $\varphi\in\CU$,
  where $\CU$ is as in Case (ii). Such a map exists since $\CU$ is
  $C^\infty$-open and the set of non-degenerate Hamiltonian
  diffeomorphisms is $C^\infty$-dense (and open). We will show that
  there exists $\delta>0$ such that $\ugamma(\psi)>\delta$ for all
  $\psi\in \CU$ which are $C^\infty$-close to $\varphi$. Hence, in
  this case, we can take a small $C^\infty$-neighborhood of $\varphi$
  as $\CW$.

\begin{Lemma}
  \label{lemma:case2}
  Let $(M, \omega)$ be a closed symplectic manifold. Suppose that
  there exists a $C^\infty$-open $\CU \subset \Ham(M, \omega)$ such
  that all $\varphi \in \CU$ have at most $q=\dim H_*(M)$ periodic
  points. Then the function $\ugamma \colon \CU \to [0, \infty)$ is
  locally uniformly bounded away from zero at every non-degenerate
  $\varphi \in \CU$.
  \end{Lemma}

  Note that the proof of Theorem \ref{thm:generic} will be completed
  once we prove Lemma \ref{lemma:case2}. To prove the lemma, arguing
  by contradiction, fix a non-degenerate $\varphi \in \CU$ and assume
  that there exists a sequence $\psi_i\to \varphi$ in $\CU$ such that
  \[
    \ugamma(\psi_i)\to 0.
  \]
  Here and below convergence of maps is always understood in the
  $C^\infty$-sense.

  We claim that when $i$ is large enough, all periodic points of
  $\psi_i$ are close to periodic points of $\varphi$, and hence there
  exists a closed ball $V\subset M$ containing no periodic points of
  any of these maps. Indeed, since $\varphi$ is non-degenerate and
  \[
    |\Fix(\varphi)|\leq |\Per(\varphi)|\leq q=\dim \H_*(M),
  \]
  by the Arnold conjecture (see \cite{FO, LT} and also \cite{Pa}),
  \[
    \Per(\varphi)=\Fix(\varphi) \textrm{ and }
    |\Per(\varphi)|=|\Fix(\varphi)|=q.
  \]
  Furthermore, when $i$ is large enough, $\psi_i\in\CU$ is also
  non-degenerate since $\psi_i\to \varphi$. Therefore, again by the
  Arnold conjecture,
  \[
    \Per(\psi_i)=\Fix(\psi_i) \textrm{ and }
    |\Per(\psi_i)|=|\Fix(\psi_i)|=q.
  \]
  It follows that $\Per(\psi_i)$ converges to $\Per(\varphi)$.

  Next, take $G\geq 0$ as in Lemma \ref{lemma:scl}, which is supported
  in $V$ and small enough so that $\varphi\varphi_G\in \CU$. Hence,
  $\psi_i\varphi_G\in \CU$ when $i$ is large; for $\psi_i\to \varphi$
  and thus $\psi_i\varphi_G\to \varphi\varphi_G$. On the other hand,
  due to the assumption that $\ugamma(\psi_i)\to 0$, we have
  \[
    \s(G)>\ugamma(\psi_i),
  \]
  when again $i$ is sufficiently large. By the strong closing lemma,
  the composition $\psi_i\varphi_G$ has a periodic orbit passing
  through $V$. On the other hand, the fixed points of $\psi_i$ (or
  equivalently the periodic points) are among the fixed points of
  $\psi_i\varphi$ because $\supp G\subset V$. It follows that
  \[
    |\Per(\psi_i\varphi_G)|\geq q+1
  \]
  when $i$ is large enough, which is impossible since
  $\psi_i\varphi_G\in \CU$. This contradiction completes the proof of
  Lemma \ref{lemma:case2} and hence of Theorem \ref{thm:generic}.
\end{proof}

\begin{Remark}
  \label{rmk-cinfty}
  If in Theorem \ref{thm:generic} we were to find a $C^1$-dense (and
  open) set of Hamiltonian diffeomorphisms rather than
  $C^\infty$-dense, the argument would be considerably
  simpler. Namely, in this case it would be enough to first construct
  a map $\varphi$ with just one hyperbolic periodic point. Once this
  is done, we could apply the results from \cite{Ha, Xi} to create
  non-trivial transverse homoclinic intersections, and hence a
  horseshoe (see \cite{KH}) by a $C^1$-small perturbation. As a
  consequence, the perturbed map $\psi$ would have infinitely many
  hyperbolic periodic points. For any $m\in\N$, having at least $m$
  such points is a $C^1$-open property. Now we can take any
  $m>\dim \H_*(M)$.
\end{Remark}

\subsection{Sugimoto manifolds and further remarks}
\label{sec:Sugimoto}
As is shown in \cite{Su21}, a strongly non-degenerate Hamiltonian
diffeomorphism $\varphi$ of a closed symplectic manifold $M^{2n}$ has
either a non-hyperbolic periodic point or infinitely many hyperbolic
periodic points when $M$ meets one of the following requirements:
\begin{itemize}
   \item[(i)] $n$ is odd;
   \item[(ii)] $\H_{\scriptscriptstyle{odd}}(M)\neq 0$;
   \item[(iii)] the minimal Chern number of $M$ is greater than 1.
\end{itemize}
Below we refer to a closed symplectic manifold meeting at least one of
these requirements as a \emph{Sugimoto manifold}. For this class of
manifolds Theorem \ref{thm:generic} has a more direct proof and can be
slightly refined. We do this in two steps.

Denote by $\CV_m$, $m\in \N$, the set of Hamiltonian diffeomorphisms
with at least $m$ hyperbolic points. Note that we do not require the
elements of $\CV_m$ to be strongly non-degenerate.

\begin{Proposition}
  \label{prop:Sugimoto}
  Let $M$ be a Sugimoto manifold. Then for any $m\in \N$ the set
  $\CV_m$ is $C^1$-open and $C^\infty$-dense in the space of all
  Hamiltonian diffeomorphisms.
\end{Proposition}

\begin{proof}
  The statement that $\CV_m$ is $C^1$-open is obvious. (It is
  essential here that $m$ is finite.)  To show that it is
  $C^\infty$-dense we argue as in \cite{Su21} and the proof of Theorem
  \ref{thm:generic}. Let $\varphi$ be a Hamiltonian diffeomorphism. To
  prove the proposition, we need to find $\psi\in \CV_m$ arbitrarily
  $C^\infty$-close to $\varphi$.  Since the set of strongly
  non-degenerate Hamiltonian diffeomorphisms is $C^\infty$-dense, we
  can assume that $\varphi$ is in this class. As shown in \cite{Su21},
  $\varphi$ has infinitely many hyperbolic periodic points or a
  (non-degenerate) non-hyperbolic point. In the former case,
  $\varphi\in\CV_m$ for all $m\in\N$.  In the latter case, by
  \cite[Prop.\ 8.2]{Ar}, for any $m\in\N$ there exists $\psi\in\CV_m$
  arbitrarily close to $\varphi$.
\end{proof}

As an immediate consequence, we obtain a slightly more precise variant
of the main result from \cite{Su21}:

\begin{Corollary}
  \label{cor:Sugimoto1}
  Assume that $M$ is a Sugimoto manifold. Then $C^\infty$-generically
  a Hamiltonian diffeomorphism $\varphi$ of $M$ has infinitely many
  hyperbolic periodic points.
\end{Corollary}

The key difference with \cite{Su21} is that the periodic points of
$\varphi$ here are specified to be hyperbolic. The residual set in
this corollary is, of course,
\[
  \CV:=\bigcap_{m\in \N}\CV_m.
\]
We note that this set is not $C^1$- and even $C^\infty$-open. However,
one can require in addition $\varphi$ to be strongly
non-degenerate. Indeed, the set of such maps is residual and its
intersection with $\CV$ is still a residual set.

Closer to the immediate subject of the paper we have the following
refinement of Theorem \ref{thm:generic} and Corollary
\ref{cor:sphere}:

\begin{Corollary}
  \label{cor:Sugimoto2}
  Assume that $M$ is a Sugimoto manifold. Then $\ugamma$ is locally
  uniformly bounded away from zero on a $C^1$-open and
  $C^\infty$-dense set of Hamiltonian diffeomorphisms of $M$.
\end{Corollary}

Here we can take any $\CV_m$ with $m>\dim\H_*(M)$ as a $C^1$-open and
$C^\infty$-dense set, where $\ugamma$ is locally uniformly bounded
away from zero. Note also that in this corollary we can again replace
the spectral norm by the Hofer norm.

\begin{Remark}
  In contrast with Theorem \ref{thm:generic}, $C^\infty$-generic
  existence of infinitely many periodic points is not known to hold
  without some additional assumptions on $M$. The class of Sugimoto
  manifolds is the broadest to date for which such existence has been
  proved, \cite{Su21}. (See also \cite{GG:generic} for the original
  result and a different approach.)
\end{Remark}

\begin{Remark}
  \label{rmk:gamma-unbounded}
  Continuing the discussion from the introduction and Remark
  \ref{rmk:behavior}, we give here some ``textbook'' examples where
  $\gamma\big(\varphi^k\big)$ grows linearly, and hence
  $\ugamma(\varphi)=\infty$, and at the same time all periodic points
  of $\varphi$ are fixed points:
  $\Per(\varphi)=\Fix(\varphi)$. Namely, let $H\colon M\to \R$ be a
  non-constant autonomous Hamiltonian such that $H$ has only finitely
  many critical values and all non-constant periodic orbits of the
  flow of $H$ are non-contractible. Set $\varphi=\varphi_H$. Then, as
  is easy to see, $\gamma\big(\varphi^k\big)$ grows linearly and the
  only periodic points of $\varphi$ are the critical points of
  $H$. For instance, we can take $H=\sin(2\pi \theta)$, where $\theta$
  is the first angular coordinate $\theta$ on
  $\T^2=\R^2/\Z^2$. Alternatively, let $(\T^4,\omega)$ be a Zehnder's
  torus, i.e., a torus equipped with a sufficiently irrational
  translation invariant symplectic structure $\omega$ (see \cite{Ze}),
  and again let $\theta\colon \T^4\to \R/\Z$ be a fixed angular
  coordinate. Then the flow of $H$ given by the same formula has no
  periodic orbits at all, contractible or not, other than the critical
  points of $H$: the 3-dimensional tori $\theta=1/2$ and
  $\theta=3/2$. In both cases, $\gamma\big(\varphi^k\big)=2k$.  More
  surprisingly, there exists a Hamiltonian diffeomorphism
  $\varphi\colon S^2\times S^2\to S^2\times S^2$ such that
  $\gamma\big(\varphi^k\big)$ grows linearly; see \cite[Rmk.\
  8]{Sh:HZ} and \cite[Thm.\ 6.2.6]{PR}, although the argument is quite
  indirect.

  In all these examples, $\dim \H_*(\Fix(\varphi))=\dim \H_*(M)$ over
  any field, in addition to the condition that
  $\Per(\varphi)=\Fix(\varphi)$. Loosely following \cite{AS}, we call
  such a map $\varphi$ a \emph{generalized
    pseudo-rotation}. Generalized pseudo-rotations from the above
  examples have simple dynamics. However, this is not necessarily so
  in general. For instance, in dimension six and higher Morse-Bott
  non-degenerate, generalized pseudo-rotations $\varphi$ with positive
  topological entropy have been recently constructed in
  \cite{Ci:new}. Such a generalized pseudo-rotation can be neither a
  $C^0$-a.i.\ since $\htop(\varphi)>0$ (see \cite{A-Z}) nor a
  $\gamma$-a.i. In fact, $\gamma\big(\varphi^k\big)$ also grows
  linearly since $M$ is aspherical and $\Per(\varphi)=\Fix(\varphi)$
  has finitely many connected components.
\end{Remark}

\end{document}